\documentclass[12pt]{article}
\usepackage[latin9]{inputenc}
\usepackage{geometry}
\usepackage{color}
\usepackage{float}
\usepackage{mathtools}
\usepackage{amsmath}
\usepackage{amsthm}
\usepackage{amssymb}
\usepackage{cancel}
\usepackage{microtype}
\usepackage[bookmarks=false,
 breaklinks=false,pdfborder={0 0 1},backref=section,colorlinks=true]
 {hyperref}
\hypersetup{pdftitle={Fixed Points},
 pagebackref=true,plainpages=false,unicode=false,pdftoolbar=true,pdfmenubar=true,pdffitwindow=false,pdfstartview={FitH},pdfnewwindow=true,linkcolor=blue,citecolor=red,filecolor=magenta,urlcolor=cyan}

\makeatletter

\floatstyle{ruled}
\newfloat{algorithm}{tbp}{loa}
\providecommand{\algorithmname}{Algorithm}
\floatname{algorithm}{\protect\algorithmname}

\newcommand{\lyxaddress}[1]{
	\par {\raggedright #1
	\vspace{1.4em}
	\noindent\par}
}
\theoremstyle{plain}
\newtheorem{thm}{\protect\theoremname}[section]
\theoremstyle{definition}
\newtheorem{defn}[thm]{\protect\definitionname}
\theoremstyle{plain}
\newtheorem{lyxalgorithm}[thm]{\protect\algorithmname}
\theoremstyle{definition}
\newtheorem{condition}[thm]{\protect\conditionname}
\theoremstyle{plain}
\newtheorem{lem}[thm]{\protect\lemmaname}
\theoremstyle{remark}
\newtheorem{rem}[thm]{\protect\remarkname}
\theoremstyle{plain}
\newtheorem{cor}[thm]{\protect\corollaryname}
\theoremstyle{definition}
\newtheorem{example}[thm]{\protect\examplename}

\usepackage[noadjust]{cite}

\numberwithin {equation}{section} 
\numberwithin {figure}{section} 
\usepackage{paralist}

\theoremstyle {plain}      
\newtheorem*{lem*}{\lemmaname }\theoremstyle {definition}      
\newtheorem*{rem*}{\remarkname }\theoremstyle {plain}      
\newtheorem*{cor*}{\corollaryname }\usepackage{cite}
\usepackage{graphicx}
\usepackage{graphpap}
\usepackage{ifthen}
\usepackage{mathrsfs}
\usepackage{enumitem}
\usepackage{pstricks}
\usepackage{algorithmic}
\usepackage{fancyhdr}
\usepackage[us,12hr]{datetime}
\usepackage{exscale}
\usepackage{tabularx}
\usepackage{latexsym}
\usepackage{verbatim}
\usepackage{multirow}
\usepackage{comment}
\usepackage{lscape}
\usepackage{booktabs}
\usepackage{algorithm}
\usepackage{algorithmic}
\usepackage{mathrsfs}

\numberwithin {equation}{section} 
\numberwithin {table}{section} 
\numberwithin {algorithm}{section} 
\usepackage[noabbrev]{cleveref}

\usepackage{refcount}

\crefname {thm}{Theorem}{Theorems} \Crefname {thm}{Theorem}{Theorems} \crefname {problem}{Problem}{Theorems}
\Crefname {problem}{Problem}{Theorems} \Crefname {assump}{Assumption}{Theorems}
\crefname {assump}{Assumption}{Theorems} \crefname {conjecture}{Conjecture}{Theorems}
\Crefname {conjecture}{Conjecture}{Theorems} \crefname {prop}{Proposition}{Propositions}
\Crefname {prop}{Proposition}{Propositions} \crefname {cor}{Corollary}{Corollaries}
\Crefname {cor}{Corollary}{Corollaries} \crefname {lem}{Lemma}{Lemmas} \Crefname {lem}{Lemma}{Lemmas}
\theoremstyle {definition} \crefname {defn}{definition}{definitions} \Crefname {defn}{Definition}{Definitions}
\crefname {conj}{Conjecture}{Conjectures} \Crefname {conj}{Conjecture}{Conjectures}
\crefname {remark}{Remark}{Remarks} \Crefname {remark}{Remark}{Remarks}
\crefname {rmk}{Remark}{Remarks} \Crefname {rmk}{Remark}{Remarks} \crefname {example}{Example}{Examples}
\Crefname {example}{Example}{Examples} \Crefname {case}{Case}{Cases} \Crefname {Case}{Case}{Cases}
\crefname {align}{}{} \Crefname {align}{}{} \crefname {equation}{}{} \Crefname {equation}{}{}

{}   

\usepackage{mdwlist}

\usepackage[misc]{ifsym}

\usepackage[normalem]{ulem}
\usepackage{enumitem}

\providecommand{\algorithmname}{Algorithm}
\providecommand{\conditionname}{Condition}
\providecommand{\corollaryname}{Corollary}
\providecommand{\definitionname}{Definition}
\providecommand{\examplename}{Example}
\providecommand{\lemmaname}{Lemma}
\providecommand{\remarkname}{Remark}
\providecommand{\theoremname}{Theorem}

\makeatother

\providecommand{\algorithmname}{Algorithm}
\providecommand{\conditionname}{Condition}
\providecommand{\corollaryname}{Corollary}
\providecommand{\definitionname}{Definition}
\providecommand{\examplename}{Example}
\providecommand{\lemmaname}{Lemma}
\providecommand{\remarkname}{Remark}
\providecommand{\theoremname}{Theorem}

\begin{document}
\title{\textbf{A necessary condition for the guarantee of the superiorization
method}}
\author{Kay Barshad$^{2,1}$, Yair Censor$^{1}$, Walaa Moursi$^{2}$, Tyler
Weames$^{2}$ \\
 and Henry Wolkowicz$^{2}$}
\date{July 13, 2024. Revised: January 23, 2025.}
\maketitle

\lyxaddress{$^{1}$Department of Mathematics, University of Haifa, Mt. Carmel,
Haifa 3498838, Israel \\
 \Letter ~ \href{mailto:yair@math.haifa.ac.il}{yair@math.haifa.ac.il}}

\lyxaddress{$^{2}$Department of Combinatorics and Optimization, Faculty of Mathematics,
University of Waterloo, Waterloo, Ontario, Canada N2L 3G1\\
 \Letter ~ \href{mailto:kbarshad@uwaterloo.ca}{kbarshad@uwaterloo.ca};
\href{mailto:walaa.moursi@uwaterloo.ca}{walaa.moursi@uwaterloo.ca};
\href{mailto:tweames@uwaterloo.ca}{tweames@uwaterloo.ca}; \href{mailto:hwolkowi@uwaterloo.ca}{hwolkowi@uwaterloo.ca}}
\begin{abstract}
We study a method that involves principally convex feasibility-seeking
and makes secondary efforts of objective function value reduction.
This is the well-known superiorization method (SM), where the iterates
of an asymptotically convergent iterative feasibility-seeking algorithm
are perturbed by objective function nonascent steps. We investigate
the question under what conditions a sequence generated by an SM algorithm
asymptotically converges to a feasible point whose objective function
value is superior (meaning smaller or equal) to that of a feasible
point reached by the corresponding unperturbed one (i.e., the exactly
same feasibility-seeking algorithm that the SM algorithm employs.)
This question is yet only partially answered in the literature. We
present a condition under which an SM algorithm that uses negative
gradient descent steps in its perturbations fails to yield such a
superior outcome. The significance of the discovery of this ``negative
condition'' is that it necessitates that the inverse of this condition
will have to be assumed to hold in any future guarantee result for
the SM. The condition is important for practitioners who use the SM
because it is avoidable in experimental work with the SM, thus increasing
the success rate of the method in real-world applications. 
\end{abstract}

\section{Introduction}\label{sec:Introduction}

The superiorization method (SM) interlaces objective function value
reduction steps (called ``perturbations'') into a feasibility-seeking
algorithm (called the ``basic algorithm''), creating a, so called,
``superiorized version of the basic algorithm''. These steps cause
the objective function to reach lower values locally, prior to performing
the next feasibility-seeking iterations. A mathematical guarantee
has not been found to date that the overall process of the superiorized
version of the basic algorithm will not only retain its feasibility-seeking
nature, but also accumulate and preserve globally the objective function
reductions. For more information concerning the SM see, for example,
\cite{SM-bib-page}, \cite{assymetric-review}, \cite{cz14-feje},
\cite{ert-salkim-2023}, \cite{jmmtshsart} and \cite{belen-thesis}.

Numerous works that are cited in \cite{SM-bib-page} show that this
global function reduction of the SM occurs in practice in many real-world
applications. In addition to a partial answer to the guarantee of
lower objective function values given in \cite{censor-levy-2019}
with the aid of \textbf{the concentration of measure principle}, there
is also the partial result of \cite[Theorem 4.1]{cz14-feje} about
\textbf{strict Fejér monotonicity} of sequences generated by an SM
algorithm.

In the SM, any sequence generated by the superiorized version of a
bounded perturbation resilient (see Definition \ref{def:resilient}
below) iterative feasibility-seeking algorithm converges to a feasible
point. The ``guarantee question of the SM'' is to find conditions
that assure that the objective function value at this point is smaller
or equal to that of a point to which the SM algorithm would have converged
if no perturbations were applied, everything else being equal.

In the preliminaries Section \ref{sect:prels} we provide, for the
reader's convenience, a compact brief review about the SM and bounded
perturbations resilience of algorithms and present the Dynamic String-Averaging
projection (DSAP) feasibility-seeking algorithmic scheme of \cite{cz14-feje}
that is used in this paper. Recent reviews on the topic appear in
\cite{herman-jano-2020} and \cite{assymetric-review}.

In spite of all the above positive statements, examples of cases where
the SM fails have been constructed, see \cite{Aratcho-2023,belen-thesis}.
In this context, ``fails'' means that the guarantee question stated
above did not hold. So, the quest for recognizing the properties of
such situations in order to make statements on the guarantee problem
of the SM continues.

Proving mathematically a guarantee of global objective function value
reduction of the SM, compared to running its feasibility-seeking algorithm
without perturbations, will probably require some additional assumptions
on the feasible set. Such assumptions will apparently refer to the
objective function, the parameters involved, or even to the set of
permissible initialization points. We present in this note a ``negative
condition'', namely, a condition under which an SM algorithm that
is based on the Dynamic String-Averaging Projection (DSAP) feasibility-seeking
algorithmic scheme along with negative gradient descent steps in its
perturbations, will fail to yield a superior outcome.

The significance of this negative condition is twofold. On one hand,
it necessitates that a reverse statement that will nullify it will
have to be assumed to hold in any potential future guarantee result
for the SM. On the other hand, we show that when this condition does
hold, then the SM algorithm can produce a superior result only if
it converges to an optimal point. The latter statement is related
to the alternative in Theorem 4.1(a) of \cite{cz14-feje}, reproduced
in Theorem \ref{thm:fejer} in Section \ref{sec:prelim} below.

Although this note discusses theory, we care to make a comment about
computational aspects. In numerical examples one should present large-size
problems because for small problems many algorithms ``work'' and
performance improvements start to show when the problems are very
large and quite sparse. We care to mention this because this is why
the SM was envisioned in the first place. For example, in \cite{LinSup-2017}
it is shown that the advantage of an SM algorithm (called there ``LinSup''
for Linear Superiorization) for the data of a linear program, over
an LP algorithm (Malab's Simplex) increases as the problem sizes grow.
Many papers referenced in the continuously updated bibliography page
on the SM and perturbation resilience of algorithms \cite{SM-bib-page}
attest to the success of SM algorithms in a variety of real-world
practical large-size problems.

The original motivation of the SM was, and still is, to handle situations
in which reaching a feasible point in the nonempty intersection of
finitely many sets, in a tractable manner, is the principal task.
In such situations the SM provides a low cost perturbation that preserves
the convergence of the perturbed feasibility-seeking algorithm to
a feasible point while aiming to improve, aka superiorize, (reduce,
not necessarily optimize) the value of an objective function. Methods
to increase the effectiveness of the objective function reduction
steps in SM algorithms, without invalidating the bounded perturbation
resilience of the embedded feasibility-seeking algorithm, have recently
been published (see, for instance, \cite{Aratcho-2023}, \cite{Pakkaranang-2020},
\cite{censor-levy-2019} and \cite{ert-salkim-2023}). These include
a randomization approach and an approach with restarts of the step-sizes. 

Support for the reasoning of the SM may be borrowed from the American
scientist and Noble-laureate Herbert Simon who was in favor of \textquotedblleft satisficing\textquotedblright{}
rather than \textquotedblleft maximizing\textquotedblright . Satisficing
is a decision-making strategy that aims for a satisfactory or adequate
result, rather than the optimal solution. This is because aiming for
the optimal solution may necessitate needless expenditure of time,
energy and resources. The term \textquotedblleft satisfice\textquotedblright{}
was coined by Herbert Simon in 1956 \cite{Simon56}, see also: \href{https://en.wikipedia.org/wiki/Satisficing}{https://en.wikipedia.org/wiki/Satisficing}.

The paper is structured as follows. In Section \ref{sec:Some-applications}
we list a few (out of many) works that used the SM in practice and
in Section \ref{sec:prelim} we provide, for the reader's convenience,
a compact brief review about the SM and bounded perturbations resilience
of algorithms. Recent reviews appear in \cite{herman-jano-2020} and
\cite{assymetric-review}. In Section \ref{sec:prelim} we also summarize
the Dynamic String-Averaging Projection (DSAP) feasibility-seeking
algorithmic scheme of \cite{cz14-feje}, i.e.,~the algorithm that
is used in this paper. The negative condition and its consequences
appear in Section \ref{sec:negative}.

\section{Some applications that use the superiorization methodology}\label{sec:Some-applications}

In \cite{guenter-compare-sup-reg-2022}, Guenter et al. consider the
fully-discretized modeling of an \textbf{image reconstruction from
projections} problem that leads to a huge and very sparse system of
linear equations. Solving such systems, sometimes under limitations
on the computing resources, remains a challenge. The authors aim not
only at solving the linear system resulting from the modeling alone,
but consider the constrained optimization problem of minimizing an
objective function subject to the modeling constraints. To do so,
they recognize two fundamental approaches: (i) superiorization, and
(ii) regularization. Within these two methodological approaches they
evaluate 21 algorithms over a collection of 18 different ``phantoms''
(i.e., test problems), presenting their experimental results in very
informative ways.

In \cite{fink-2021} Fink et al. study the \textbf{nonconvex multi-group
multicast beamforming problem} with quality-of-service constraints
and per-antenna power constraints. They formulate a convex relaxation
of the problem as a semidefinite program in a real Hilbert space.
This allows them to approximate a point in the feasible set by iteratively
applying a bounded perturbation resilient fixed-point mapping. Inspired
by the superiorization methodology, they use this mapping as a basic
algorithm, and add in each iteration a small perturbation with the
intent to reduce the objective value and the distance to nonconvex
rank-constraint sets.

Pakkaranang et al. \cite{Pakkaranang-2020} construct a novel algorithm
for solving \textbf{non-smooth composite optimization problems}. By
using an inertial technique, they propose a modified proximal gradient
algorithm with outer perturbations and obtain strong convergence results
for finding a solution of a composite optimization problem. Based
on bounded perturbation resilience, they present their algorithm with
the superiorization method and apply it to image recovery problems.
They provide numerical experiments that show the efficiency of the
algorithm and compare it with previously known algorithms in signal
recovery.

Especially interesting is the recent work of Ma et al. \cite{sahinidis-2021}
who propose a novel decomposition framework for \textbf{derivative-free
optimization (DFO) algorithms} that significantly extends the scope
of current DFO solvers to larger-scale problems. They show that their
proposed framework closely relates to the superiorization methodology.

Many more publications on practical applications of the SM can be
found in \cite{SM-bib-page}; these include~\cite{jmmtshsart,janakmip,langis,lzhem,sohm},
to mention but a few.

\section{Preliminaries: superiorization and dynamic string-averaging}\label{sec:prelim}

\label{sect:prels}

In this section we provide some background concerning the superiorization
methodology and dynamic string-averaging. In Subsection \ref{subsec:SM}
we discuss the superiorization methodology and bounded perturbation
resilience. In Subsection \ref{subsec:prelim-DSAP} we recall the
convergence properties of the Dynamic String-Averaging Projection
(DSAP) feasibility-seeking algorithm and its superiorized version,
presented, respectively, in \cite{cz12} and \cite{cz14-feje}.

\subsection{The superiorization methodology}\label{subsec:SM}

The superiorization methodology (SM) was born when the terms and notions
``superiorization'' and ``perturbation resilience'', in the present
context, first appeared in the 2009 paper \cite{dhc09}. This followed
its 2007 forerunner by Butnariu et al. \cite{bdhk07}. The ideas have
some of their roots in the 2006 and 2008 papers of Butnariu et al.
\cite{brz06,ButReich07,brz08} where it was shown that if iterates
of a nonexpansive operator converge for any initial point, then its
inexact iterates with summable errors also converge. Since its inception
in 2007, the SM has evolved and gained prominence. Recent publications
on the SM devoted to either weak or strong superiorization, though
without yet using these terms, are: \cite{bdhk07,cdh10,cz12,dhc09,rand-conmath,gh13,hd08,hgdc12,ndh12,pscr10}.
Many of these contain a detailed description of the SM, its motivation,
and an up-to-date review of SM-related previous work.

The Webpage\footnote{\href{http://math.haifa.ac.il/yair/bib-superiorization-censor.html\#top}{http://math.haifa.ac.il/yair/bib-superiorization-censor.html\#top},
last updated on December 26, 2024 with 191 items.} \cite{SM-bib-page} is dedicated to superiorization and perturbation
resilience of algorithms and contains a continuously updated bibliography
on the subject. It is a source for the wealth of work done in this
field to date, including two dedicated special issues of journals
\cite{SI-inverse-prob-2017} and \cite{SI-JANO-2020}.

Our recent review \cite{assymetric-review} can serve as an introduction
to the SM; also \cite{gth-sup4IA} and \cite{herman-jano-2020} are
very helpful. Just to make the continued reading here more convenient
for the reader we give below some of the fundamental notions of the
SM.

Throughout the rest of the paper we consider a real Hilbert space
$X$ with the norm $\left\Vert \cdot\right\Vert $ and a real-valued,
convex and continuous function $\phi:X\rightarrow R$. For a point
$z\in X$, we denote by $\partial\phi\left(z\right)$ the subgradient
set of $\phi$ at $z$. For a point $x\in X$ and a nonempty, closed
and convex subset $S$ of $X$, we denote by $d\left(x,S\right)$
the distance of $x$ to $S$.

Bounded perturbation resilience is now a topic of interest in the
current literature (see, for instance, \cite{Reich_Taiwo}). We recall
its definition. 
\begin{defn}
[{\cite[Definition 1]{cdh10}}] \label{def:resilient}\textbf{Bounded
perturbation resilience (BPR)}. Let $\Gamma\subseteq X$ be a given
nonempty subset of $X$. An algorithmic operator $\mathcal{A}:X\rightarrow X$
is said to be \textbf{bounded perturbations resilient with respect
to $\Gamma$}\emph{ }if the following is true: If a sequence $\{x^{k}\}_{k=0}^{\infty},$
generated by the basic algorithm $x^{k+1}:=\mathcal{A}(x^{k}),$ for
all $k\geq0,$ converges to a point in $\Gamma$ for all $x^{0}\in X$,
then any sequence $\{y^{k}\}_{k=0}^{\infty}$ of points in $X$ that
is generated by the algorithm $y^{k+1}=\mathcal{A}(y^{k}+\beta_{k}v^{k}),$
for all $k\geq0,$ also converges to a point in $\Gamma$ for all
$y^{0}\in X$ provided that, for all $k\geq0$, $\beta_{k}v^{k}$
are \textbf{bounded perturbations}, meaning that $\beta_{k}\geq0$
for all $k\geq0$ such that $\sum_{k=0}^{\infty}\beta_{k}<\infty$,
and that the vector sequence $\{v^{k}\}_{k=0}^{\infty}$ is bounded.
\\
A basic algorithm of the form $x^{k+1}:=\mathcal{A}(x^{k}),$ for
all $k\geq0,$is said to be \textbf{bounded perturbations resilient
with respect to $\Gamma$ }if its algorithmic operator $\mathcal{A}$
is bounded perturbations resilient with respect to $\Gamma$.\\
\end{defn}

Algorithm \ref{alg:super-process}, presented here, is a superiorized
version of the basic (feasibility-seeking) algorithm governed by\textbf{
}$\mathcal{A}$.

\begin{algorithm}
\begin{lyxalgorithm}
\textbf{\label{alg:super-process}}Superiorized version of the basic
(feasibility-seeking) algorithm governed by $\mathcal{A}$.
\end{lyxalgorithm}

\textbf{\textit{(0) Initialization}}\textit{: Let $N$ be a natural
number and let $y^{0}\in X$ be an arbitrary user-chosen vector.}

\textbf{\textit{(1)}}\textit{ }\textbf{\textit{Iterative step}}\textit{:
Given a current iteration vector $y^{k}$ , pick an $N_{k}\in\{1,2,\dots,N\}$
and start an inner loop of calculations as follows:}

\textbf{\textit{(1.1) Inner loop initialization}}\textit{: Define
$y^{k,0}=y^{k}.$}

\textbf{\textit{(1.2) Inner loop step: }}\textit{Given $y^{k,n},$
as long as $n<N_{k},$ do as follows:}

\textbf{\textit{(1.2.1) }}\textit{Pick a $0<\beta_{k,n}\leq1$ in
a way that guarantees that 
\[
\sum_{k=0}^{\infty}\sum_{n=0}^{N_{k}-1}\beta_{k,n}<\infty.
\]
}

\textbf{\textit{(1.2.2)}}\textit{ Pick an ${\displaystyle s^{k,n}\in\partial\phi(y^{k,n})}$
and define direction vectors $v^{k,n}$ as follows: 
\[
v^{k,n}:=\left\{ \begin{array}{cc}
-\frac{{\displaystyle s^{k,n}}}{{\displaystyle \left\Vert s^{k,n}\right\Vert }}, & \text{if }0\notin\partial\phi(y^{k,n}),\\
0, & \text{if }0\in\partial\phi(y^{k,n}).
\end{array}\right.
\]
}

\textbf{\textit{(1.2.3) }}\textit{Calculate the perturbed iterate
\begin{equation}
y^{k,n+1}:=y^{k,n}+\beta_{k,n}v^{k,n}\label{eq:2.11}
\end{equation}
and if $n+1<N_{k}$ set $n\leftarrow n+1$ and go to }\textbf{\textit{(1.2)}}\textit{,
otherwise go to }\textbf{\textit{(1.3)}}\textit{.}

\textbf{\textit{(1.3) }}\textit{Exit the inner loop with the vector
$y^{k,N_{k}}$}

\textbf{\textit{(1.4) }}\textit{Calculate 
\begin{equation}
y^{k+1}:=\mathcal{A}(y^{k,N_{k}})\label{eq:-6}
\end{equation}
set $k\leftarrow k+1$ and go back to }\textbf{\textit{(1)}}\textit{.} 
\end{algorithm}

\subsection{The convergence properties of the DSAP algorithm and its superiorized
version}\label{subsec:prelim-DSAP}

The Dynamic String-Averaging Projection (DSAP) method of \cite{cz12}
constitutes a family of algorithmic operators that can play the role
of $\mathcal{A}$ in Algorithm \ref{alg:super-process}. For each
$i=1,2,\dots,m,$ let $C_{i}$ be a nonempty, closed and convex subset
of $X$ and denote by $P_{i}:=P_{C_{i}}$ the metric projection onto
the set $C_{i}.$ An \textbf{index vector} is a vector $t=(t_{1},t_{2},\dots,t_{q})$
such that $t_{s}\in\{1,2,\dots,m\}$ for all $s=1,2,\dots,q$, whose
length is $\ell(t)=q.$ The composition of the individual projections
onto the sets whose indices appear in the index vector $t$ is $P[t]:=P_{t_{q}}\cdots P_{t_{2}}P_{t_{1}}$,
called a \textbf{string operator}.

A finite set $\Omega$ of index vectors is called \textbf{fit} if
for each $i\in\{1,2,\dots,m\}$, there exists a vector $t=(t_{1},t_{2},\dots,t_{q})\in\Omega$
such that $t_{s}=i$ for some $s\in\{1,2,\dots,q\}$. Denote by $\mathcal{M}$
the collection of all pairs $(\Omega,w)$, where $\Omega$ is a finite
fit set of index vectors and $w:\Omega\rightarrow(0,\infty)$ is such
that $\sum_{t\in\Omega}w(t)=1.$

For any $(\Omega,w)\in\mathcal{M}$ define the convex combination
of the end-points of all strings defined by members of $\Omega$ 
\[
P_{\Omega,w}\left(x\right):=\sum_{t\in\Omega}w\left(t\right)P\left[t\right]\left(x\right)
\]
for each $x\in X$. Let $\Delta\in(0,1/m)$, and fix an arbitrary
integer $\bar{q}\geq m.$ Denote by $\mathcal{M}_{\ast}\equiv\mathcal{M}_{\ast}(\Delta,\bar{q})$
the set of all $(\Omega,w)\in\mathcal{M}$ such that the lengths of
the strings are bounded and the weights are all bounded away from
zero, i.e., 
\[
\mathcal{M}_{\ast}:=\{(\Omega,w)\in\mathcal{M\mid}\text{ }\ell(t)\leq\bar{q}\text{ and }w(t)\geq\Delta,\text{ }\forall\text{ }t\in\Omega\}.
\]

The convergence properties and bounded perturbation resilience of
the DSAP method were analyzed in \cite{cz12}.

\begin{algorithm}
\begin{lyxalgorithm}
\label{alg:DSAP}\textbf{ The DSAP method with variable strings and
variable weights}
\end{lyxalgorithm}

\textbf{\textit{Initialization}}\textit{: Select an arbitrary $x^{0}\in X$,}

\textbf{\textit{Iterative step}}\textit{: Given a current iteration
vector $x^{k}$ pick a pair $(\Omega_{k},w_{k})\in\mathcal{M}_{\ast}$
and calculate the next iteration vector $x^{k+1}$ by 
\[
x^{k+1}:=P_{\Omega_{k},w_{k}}(x^{k})\text{.}
\]
} 
\end{algorithm}

The strong convergences of Algorithm \ref{alg:DSAP} and its superiorized
version, as demonstrated in Theorems \ref{thm:=00003D000020dsap}
and \ref{thm:fejer} below, are based on the following bounded regularity
condition, see \cite[Definition 5.1]{BC96}. This condition is always
satisfied in the case where the space $X$ is finite dimensional (see
\cite[Proposition 5.4]{BC96}). 
\begin{condition}
\label{bound_reg} Let $\{C_{i}\}{}_{i=1}^{m}$ be a family of nonempty,
closed and convex subsets of $X$ with a nonempty intersection $C$.
For each $\varepsilon>0$ and $M>0$, there exists $\delta>0$ such
that for each $x\in X$, we have the following implication
\[
\left\{ \begin{array}{c}
\left\Vert x\right\Vert <M\\
d\left(x,C_{i}\right)<\delta\,\mathrm{\,for\,\,all\,\,}i=1,2,\dots,m
\end{array}\Longrightarrow d\left(x,C\right)<\varepsilon.\right.
\]
\end{condition}

\begin{thm}[{{\cite[Theorem 12]{cz12}}}]
\label{thm:=00003D000020dsap}Assume that $\{C_{i}\}{}_{i=1}^{m}$
is a family of nonempty, closed and convex subsets of $X$ with a
nonempty intersection $C$, satisfying Condition \ref{bound_reg}.
Let $\left\{ \beta_{k}\right\} _{k=0}^{\infty}$ be a sequence of
non-negative numbers such that $\sum_{k=0}^{\infty}\beta_{k}<\infty$,
let $\left\{ v^{k}\right\} _{k=0}^{\infty}\subset X$ be a norm-bounded
sequence, let $\left\{ (\Omega_{k},w_{k})\right\} _{k=0}^{\infty}\in\mathcal{M}_{\ast}$
for all $k=0,1,\dots$. Then any sequence $\left\{ y^{k}\right\} _{k=0}^{\infty}$,
generated by the iterative formula 
\[
y^{k+1}:=P_{\Omega_{k},w_{k}}\left(y^{k}+\beta_{k}v^{k}\right),
\]
converges in the norm of $X$, and its limit belongs to $C$. That
is, Algorithm \ref{alg:DSAP} converges to a point in $C$ and its
algorithmic operator is bounded perturbation resilient with respect
to $C$. 
\end{thm}

Using the algorithmic operator of the DSAP feasibility-seeking Algorithm
\ref{alg:DSAP} as the algorithmic operator $\mathcal{A}$ in Algorithm
\ref{alg:super-process}, we recover the following main Theorem 4.1
of \cite{cz14-feje}. 
\begin{thm}[{{\cite[Theorem 4.1]{cz14-feje}}}]
\label{thm:fejer} Assume that $\{C_{i}\}{}_{i=1}^{m}$ is a family
of nonempty, closed and convex subsets of $X$ with a nonempty intersection
$C$, satisfying Condition \ref{bound_reg}. Set $C_{\min}:=\{x\in C\mid\;\phi(x)\leq\phi(y){\text{ for all }}y\in C$.
Let $C_{\ast}\subseteq C_{\min}$ be a nonempty subset of $C_{\min}$,
let $r_{0}\in(0,1]$ and $\bar{L}\geq1$ be such that, ${\text{for all }}x\in C_{\ast}{\text{ and all }}y$
such that$\;||x-y||<r_{0},$ 
\[
|\phi(x)-\phi(y)|\leq\bar{L}||x-y||.
\]
Further, suppose that $\{(\Omega_{k},w_{k})\}_{k=0}^{\infty}\subset\mathcal{M}_{\ast}.$
Then any sequence $\{y^{k}\}_{k=0}^{\infty},$ generated by Algorithm
\ref{alg:super-process}, the superiorized version of the DSAP algorithm,
converges in the norm of $X$ to a $y^{\ast}\in C$. And exactly one
of the following two alternatives holds: 
\begin{enumerate}
\item \label{item:altone}$y^{\ast}\in C_{\min}$. 
\item \label{item:alttwo}$y^{\ast}\notin C_{\min}$ and there exist a natural
number $k_{0}$ and a $c_{0}\in(0,1)$ such that for each $x\in C_{\ast}$
and for each integer $k\geq k_{0}$, 
\[
\Vert y^{k+1}-x\Vert^{2}\leq\Vert y^{k}-x\Vert^{2}-c_{0}\sum_{n=1}^{N_{k}-1}\beta_{k,n}.
\]
\end{enumerate}
\end{thm}

This shows that $\{y^{k}\}_{k=k_{0}}^{\infty}$ is strictly Fejér-monotone
with respect to\textbf{ }$C_{\ast},$ i.e., since $c_{0}\sum_{n=1}^{N_{k}-1}\beta_{k,n}>0$,
we get that for every $x\in C_{\ast}$, the inequality $\Vert y^{k+1}-x\Vert^{2}<\Vert y^{k}-x\Vert^{2}$
holds for all $k\geq k_{0}.$ The strict Fejér-monotonicity however
does not guarantee convergence to a constrained minimum point. Rather,
it only says that the so-created feasibility-seeking sequence $\{y^{k}\}_{k=0}^{\infty}$
has the additional property of getting strictly closer, without necessarily
converging, to the points of a subset of the solution set of the constrained
minimization problem.

Published experimental results repeatedly confirm that global reduction
of the value of the objective function $\phi$ is indeed achieved,
without losing the convergence toward feasibility, see \cite{bdhk07,cdh10,cz12,dhc09,rand-conmath,gh13,hd08,hgdc12,ndh12,pscr10}.
In some of these cases the SM returns a lower value of the objective
function $\phi$ than an exact minimization method with which it is
compared, e.g., \cite{cdhst14}.

\section{The \textquotedblleft negative condition\textquotedblright{} on the
superiorization method}\label{sec:negative}

We consider the Dynamic String-Averaging Projection (DSAP) feasibility-seeking
algorithmic scheme of \cite{cz14-feje} and its superiorized version,
as presented above in Subsection \ref{subsec:prelim-DSAP}. Speaking
specifically about \cite[Algorithm 4.1]{cz14-feje}, we know that
under the assumptions of Theorem \ref{thm:fejer}, exactly one of
two things must happen, i.e., alternative \textit{\ref{item:altone}}
or \textit{\ref{item:alttwo}}. This is a non-constructive theorem
because it tells nothing about when each of the alternatives can occur.

If alternative \textit{\ref{item:altone} }holding is the case, then
it is correct to say that the sequence generated by the superiorized
version of the bounded perturbation resilient feasibility-seeking
Algorithm \ref{alg:DSAP} converges to a feasible point that has objective
function value smaller or equal to that of a point to which this algorithm
would converge if no perturbations were applied. This would be true
in this case because $y^{*}\in C_{\min}$ and the feasibility-seeking
algorithm cannot do better.

The question remains for the case of alternative \textit{\ref{item:alttwo}}.
Namely, can we give conditions under which, if alternative \textit{\ref{item:alttwo}
}holds, we will have that $\phi(y^{*})\leqslant\phi(x^{*})$, where
$x^{*}$ is the limit of the same feasibility-seeking algorithm that
is used in this SM algorithm when it is run without perturbations.
This rerun without perturbations must have everything else equal,
such as initialization point, relaxation parameters, order of constraints
within the sweeps, etc.

The desire to distinguish between the alternatives \textit{\ref{item:altone}
}and \textit{\ref{item:alttwo}} of Theorem \ref{thm:fejer} leads
us to the next lemma which gives a condition under which a limiting
feasible point cannot belong to the solution set of the constrained
minimization problem ${\rm min}\:\{\phi\left(x\right)\mid x\in C\}.$
This is the ``negative condition'' eluded to above, under which
an SM algorithm that is based on the Dynamic String-Averaging Projection
(DSAP) feasibility-seeking algorithmic scheme, which uses negative
subgradient steps in its perturbations, will fail to yield a superior
outcome. 
\begin{lem}
\label{lem:divergeSM} Let $\{C_{i}\}{}_{i=1}^{m}$ be a family of
nonempty, closed and convex subsets of $X$ with a nonempty intersection
$C$. Let $D\subset C$ be a nonempty subset of $C$, and let $r\ge1$
be a real number. Assume that $\left\{ y^{k}\right\} _{k=0}^{\infty}$
is any sequence generated by Algorithm \ref{alg:super-process}, with
positive step-sizes $\left\{ \left\{ \beta_{k,n}\right\} _{n=0}^{N_{k}-1}\right\} _{k=0}^{\infty}$,
direction vectors $\left\{ \left\{ v^{k,n}\right\} _{n=0}^{N_{k}-1}\right\} _{k=0}^{\infty}$,
and with $y^{0}\in X$ an arbitrary initialization point. Suppose
that the family $\{C_{i}\}{}_{i=1}^{m}$ satisfies Condition \ref{bound_reg}.
Then
\begin{enumerate}
\item \label{item:limitone} If $\hat{c}\in C$ is a point such that 
\begin{equation}
\left\Vert y^{0}-\hat{c}\right\Vert \le\left(r-1\right)\sum_{n=0}^{N_{0}-1}\beta_{0,n},\label{eq:}
\end{equation}
then the sequence $\left\{ y^{k}\right\} _{k=0}^{\infty}$ satisfies
\begin{equation}
\left\Vert y^{k}-\hat{c}\right\Vert \le r\sum_{\ell=0}^{k-1}\sum_{n=0}^{N_{\ell}-1}\beta_{\ell,n}\label{eq:claim-i}
\end{equation}
\textup{for all $k\geq1$.}
\item \label{item:limittwo} If, additionally, 
\begin{equation}
d\left(\hat{c},D\right)>r\sum_{\ell=0}^{\infty}\sum_{n=0}^{N_{\ell}-1}\beta_{\ell,n}\label{eq:additionally}
\end{equation}
holds, then the limit point $y^{*}$ of the sequence $\left\{ y^{k}\right\} _{k=0}^{\infty}$
(which exists by Theorem \ref{thm:fejer}) satisfies $y^{*}\notin D.$ 
\end{enumerate}
\end{lem}

\begin{proof}
\textit{\ref{item:limitone}} The proof is by induction on $k$. For
$k=1$, we use (\ref{eq:2.11}), (\ref{eq:-6}), \eqref{eq:}, the
nonexpansivity of the operator $P_{\Omega_{0},w_{0}}$ and the boundedness
of the sequence $\left\{ v^{0,n}\right\} _{n=0}^{\infty}$, to obtain
(since $\hat{c}\in C$)

\begin{align*}
\left\Vert y^{k}-\hat{c}\right\Vert  & =\left\Vert y^{1}-\hat{c}\right\Vert =\left\Vert P_{\Omega_{0},w_{0}}\left(y^{0,N_{0}}\right)-P_{\Omega_{0},w_{0}}\left(\hat{c}\right)\right\Vert \\
 & \le\left\Vert y^{0}+\sum_{n=0}^{N_{0}-1}\beta_{0,n}v^{0,n}-\hat{c}\right\Vert \le\left\Vert y^{0}-\hat{c}\right\Vert +\left\Vert \sum_{n=0}^{N_{0}-1}\beta_{0,n}v^{0,n}\right\Vert \\
 & \le\left\Vert y^{0}-\hat{c}\right\Vert +\sum_{n=0}^{N_{0}-1}\beta_{0,n}\le r\sum_{n=0}^{N_{0}-1}\beta_{0,n}=r\sum_{\ell=0}^{k-1}\sum_{n=0}^{N_{\ell}-1}\beta_{\ell,n}.
\end{align*}
Suppose that $k>1$ and make the inductive assumption that 
\begin{equation}
\left\Vert y^{k-1}-\hat{c}\right\Vert \le r\sum_{\ell=0}^{k-2}\sum_{n=0}^{N_{\ell}-1}\beta_{\ell,n}.\label{eq:-2}
\end{equation}
Since $\hat{c}\in C$ (recalling (\ref{eq:-6})), we have 
\begin{equation}
\left\Vert y^{k}-\hat{c}\right\Vert =\left\Vert P_{\Omega_{k-1},w_{k-1}}\left(y^{k-1,N_{k-1}}\right)-P_{\Omega_{k-1},w_{k-1}}\left(\hat{c}\right)\right\Vert .\label{eq:-3}
\end{equation}
By \eqref{eq:2.11}, 
\begin{equation}
y^{k-1,N_{k-1}}=y^{k-1}+\sum_{n=0}^{N_{k-1}-1}\beta_{k-1,n}v^{k-1,n}.\label{eq:-7}
\end{equation}
Since the operator $P_{\Omega_{k},w_{k}}$ is nonexpansive, we obtain
from \eqref{eq:-3} and \eqref{eq:-7} 
\begin{equation}
\left\Vert y^{k}-\hat{c}\right\Vert \leq\left\Vert y^{k-1}+\sum_{n=0}^{N_{k-1}-1}\beta_{k-1,n}v^{k-1,n}-\hat{c}\right\Vert .\label{eq:-4}
\end{equation}
By the triangle inequality and the boundedness of the sequence $\left\{ v^{k-1,n}\right\} _{n=0}^{\infty}$,
combined with \eqref{eq:-4}, we obtain 
\begin{equation}
\left\Vert y^{k}-\hat{c}\right\Vert \le\left\Vert y^{k-1}-\hat{c}\right\Vert +\sum_{n=0}^{N_{k-1}-1}\beta_{k-1,n}.\label{eq:-5}
\end{equation}
The inductive assumption \eqref{eq:-2} and inequality \eqref{eq:-5},
imply (since $r\ge1$) 
\[
\left\Vert y^{k}-\hat{c}\right\Vert \le r\sum_{\ell=0}^{k-2}\sum_{n=0}^{N_{\ell}-1}\beta_{\ell,n}+r\sum_{n=0}^{N_{k-1}-1}\beta_{k-1,n}=r\sum_{\ell=0}^{k-1}\sum_{n=0}^{N_{\ell}-1}\beta_{\ell,n}.
\]
This completes the induction and shows that \textit{ \ref{item:limitone}
} indeed holds.\\

\textit{ \ref{item:limittwo} }Assume to the contrary that $y^{*}\in D$.
Taking the limit, as $k\rightarrow\infty,$ on both sides of (\ref{eq:claim-i})
yields 
\[
\left\Vert y^{*}-\hat{c}\right\Vert \le r\sum_{\ell=0}^{\infty}\sum_{n=0}^{N_{\ell}-1}\beta_{\ell,n},
\]
which contradicts (\ref{eq:additionally}), thus proving that $y^{*}$
cannot be in $D.$ 
\end{proof}
\begin{rem}
Lemma \ref{lem:divergeSM} shows that if (\ref{eq:claim-i}) and (\ref{eq:additionally})
hold, then alternative \textit{\ref{item:altone} }of Theorem \ref{thm:fejer}
cannot hold. The case $r=1$ is of theoretical interest only because
then the initialization point $y^{0}$ is feasible, contrary to the
prevailing situation in applications wherein the feasibility-seeking
is initialized outside the feasible set $C.$\\
 
\end{rem}

The next corollary provides an insight into a necessary choice of
an initialization point of Algorithm \ref{alg:super-process} in order
to establish its convergence to a point in $D$. 
\begin{cor}
\label{cor:chioceip} Under the assumptions of Lemma \ref{lem:divergeSM},
assume that the sequence $\left\{ y^{k}\right\} _{k=0}^{\infty}$
generated by Algorithm \ref{alg:super-process} with step-sizes $\left\{ \left\{ \beta_{k,n}\right\} _{n=0}^{N_{k}-1}\right\} _{k=0}^{\infty}$
converges to a point $y^{\ast}\in D$ and there exists a point $\hat{c}\in C$
such that 
\begin{equation}
d\left(\hat{c},D\right)>\sum_{k=0}^{\infty}\sum_{n=0}^{N_{k}-1}\beta_{k,n}.\label{eq:-1}
\end{equation}
Then the following inequality holds

\begin{equation}
\left\Vert y^{0}-\hat{c}\right\Vert \ge\left(d\left(\hat{c},D\right)\left(\sum_{k=0}^{\infty}\sum_{n=0}^{N_{k}-1}\beta_{k,n}\right)^{-1}-1\right)\sum_{n=0}^{N_{0}-1}\beta_{0,n}.\label{eq:cond1-1}
\end{equation}
\end{cor}

\begin{proof}
Clearly, 
\[
d\left(\hat{c},D\right)>r\sum_{k=0}^{\infty}\sum_{n=0}^{N_{k}-1}\beta_{k,n},
\]
for any real $1\le r<d\left(\hat{c},D\right)\left(\sum_{k=0}^{\infty}\sum_{n=0}^{N_{k}-1}\beta_{k,n}\right)^{-1}$.
Since the limit point of the sequence generated by Algorithm \ref{alg:super-process}
belongs to $D$, we must have, by Lemma \ref{lem:divergeSM}, that
for each such $r$,
\[
\left\Vert y^{0}-\hat{c}\right\Vert >\left(r-1\right)\sum_{n=0}^{N_{0}-1}\beta_{0,n}.
\]
Inequality \eqref{eq:cond1-1} now follows by \eqref{eq:-1}. 
\end{proof}
\begin{rem}
Note that Lemma \ref{lem:divergeSM} and Corollary \ref{cor:chioceip}
are, in particular, true for $D=C_{\min}:=\{x\in C\mid\;\phi(x)\leq\phi(y){\text{ for all }}y\in C\}$.\\
 
\end{rem}

Under the assumption \eqref{eq:-1}, Corollary \ref{cor:chioceip}
provides a non-trivial necessary condition, namely, \eqref{eq:cond1-1},
for the convergence of any sequence generated by Algorithm \ref{alg:super-process}
to a point in $D$. However, the aforesaid condition is not sufficient
to this end, even if the assumption \eqref{eq:-1} holds. This observation
is demonstrated in the following example. 
\begin{example}
\label{exa:counterex}Let $X:=\mathbb{R}$, let $\phi:X\rightarrow X$
be defined by $\phi\left(x\right):=x^{2}$ for each $x\in X$, and
define $C:=\left[0,10\right]$. Clearly, $C$ is a closed and convex
subset of $\mathbb{R}$. Note that $\phi$ is 20-Lipschitz continuous
on $C$, that is, $\left\Vert \phi\left(x\right)-\phi\left(x^{\prime}\right)\right\Vert \le20\left\Vert x-x^{\prime}\right\Vert $
for each $x,x^{\prime}\in C$. Set $D:=C_{\min}=\left\{ 0\right\} $,
$\hat{c}:=8\in C$ and define the sequence of step-sizes $\left\{ \beta_{k,0}\right\} _{k=0}^{\infty}$
of length $1$ (that is, $N_{k}=1$ for each $k=0,1\dots$) by $\beta_{k,0}:=2^{-k}$
for each $k=0,1\dots$. We then have 
\[
d\left(\hat{c},D\right)=8>2=\sum_{k=0}^{\infty}\beta_{k,0}=\sum_{k=0}^{\infty}\sum_{n=0}^{N_{k}-1}\beta_{k,n}.
\]
Now choose any $y^{0}\ge13$. Then we have 
\[
\sum_{n=0}^{N_{0}-1}\beta_{0,n}=\beta_{0,0}=1\,\,\,{\rm and\,\,\,}\left(d\left(\hat{c},D\right)\left(\sum_{k=0}^{\infty}\sum_{n=0}^{N_{k}-1}\beta_{k,n}\right)^{-1}-1\right)=3\textcolor{red}{.}
\]
Therefore, 
\[
\left\Vert y^{0}-\hat{c}\right\Vert \ge3=\left(d\left(\hat{c},D\right)\left(\sum_{k=0}^{\infty}\sum_{n=0}^{N_{k}-1}\beta_{k,n}\right)^{-1}-1\right)\sum_{n=0}^{N_{0}-1}\beta_{0,n}.
\]
However, the sequence $\left\{ y^{k}\right\} _{k=0}^{\infty}$, generated
by Algorithm \ref{alg:super-process}, satisfies for each positive
integer $k$, 
\[
y^{k}=10-\sum_{\ell=1}^{k-1}2^{-\ell}\underset{k\rightarrow\infty}{\rightarrow}9\not\in D.
\]
Similarly, if we choose in this example $C:=\left[0,3\cdot2^{-1}\right]$,
$\hat{c}:=1\in C$, and keep all other settings the same, then Condition
\eqref{eq:-1} fails and Condition \eqref{eq:cond1-1} trivially holds.
But the sequence $\left\{ y^{k}\right\} _{k=0}^{\infty}$, generated
by Algorithm \ref{alg:super-process}, satisfies for each positive
integer $k$, 
\[
y^{k}=3\cdot2^{-1}-\sum_{\ell=1}^{k-1}2^{-\ell}\underset{k\rightarrow\infty}{\rightarrow}2^{-1}\not\in D.
\]
\\
 
\end{example}

Taking $D$ in Corollary \ref{cor:chioceip} to be a level-set of
the function $\phi$ within $C$, the next corollary follows. 
\begin{cor}
\label{cor:levelsetex} Under the assumptions of Lemma \ref{lem:divergeSM},
pick any $x^{*}\in C$ together with $y^{0}\in X$, and define $D:=\left\{ c\in C\mid\phi\left(c\right)\le\phi\left(x^{*}\right)\right\} $.
Choose positive step-sizes $\left\{ \left\{ \beta_{k,n}\right\} _{n=0}^{N_{k}-1}\right\} _{k=0}^{\infty}$
and $\hat{c}\in C$ such that 
\begin{equation}
d\left(\hat{c},D\right)>\sum_{k=0}^{\infty}\sum_{n=0}^{N_{k}-1}\beta_{k,n}\label{eq:cond1}
\end{equation}
and 
\begin{equation}
\left\Vert y^{0}-\hat{c}\right\Vert <\left(d\left(\hat{c},D\right)\left(\sum_{k=0}^{\infty}\sum_{n=0}^{N_{k}-1}\beta_{k,n}\right)^{-1}-1\right)\sum_{n=0}^{N_{0}-1}\beta_{0,n}.\label{eq:cond2}
\end{equation}
If $y^{*}$ is the limit point of a sequence $\left\{ y^{k}\right\} _{k=0}^{\infty}$
generated by Algorithm \ref{alg:super-process}, then $y^{*}\notin D$,
that is, $\phi\left(y^{*}\right)>\phi\left(x^{*}\right)$.\\
 
\end{cor}

\section{Conclusion}

In this paper we presented a condition under which a superiorization
method (SM) that uses negative gradient descent steps in its perturbations
fails to yield a superior outcome. This means that if $x^{\ast}$
is the limit point of the feasibility-seeking algorithm without perturbations,
then under conditions (presented in Corollary \ref{cor:levelsetex}),
the corresponding superiorization algorithm fails to reach a feasible
point with objective function value smaller or equal than that of
$x^{\ast}$. 

The guarantee question of the SM is, to date, only partially resolved,
but an inverse of this ``negative condition'' will have to be included
and assumed to hold in any future mathematical guarantee claim for
the SM. The condition is important for practitioners who use the SM
because it is avoidable in experimental work with the SM, thus increasing
the success rate of the method in real-world applications. In future
practical implementations of users of the SM it would be advisable
to choose the initialization point far enough from the feasible set,
so as to avoid the negative condition from occurring.

\bigskip{}

\textbf{Acknowledgments.} We thank the anonymous referee for valuable
comments and remarks that helped in improving our work. The work of
Kay Barshad and Yair Censor is supported by the Cooperation Program
in Cancer Research of the German Cancer Research Center (DKFZ) and
Israel's Ministry of Innovation, Science and Technology (MOST). The
work of Yair Censor is supported also by the ISF-NSFC joint research
plan Grant Number 2874/19 and by the U.S. National Institutes of Health
Grant Number R01CA266467. Kay Barshad, Walaa Moursi, Tyler Weames
and Henry Wolkowicz thank the Natural Sciences and Engineering Research
Council of Canada for its support.

\bigskip{}

\textbf{Data Availability.} Data sharing not applicable to this article
as no data sets were generated or analyzed during the current study.\bigskip{}

\textbf{Compliance with Ethical Standards. }The authors have no potential
conflicts of interest to declare that are relevant to the content
of this article.

\bigskip{}

\textbf{Competing Interests. }The authors have no competing interests
to declare that are relevant to the content of this article.

\bigskip{}

\end{document}